\documentclass[12pt,reqno]{article}
\usepackage{amssymb}
\usepackage{graphicx}
\usepackage{amscd}
\usepackage{fullpage}
\usepackage{graphics,amsmath,amssymb}
\usepackage{amsthm}
\usepackage{amsfonts}
\usepackage{latexsym}
\usepackage{epsf}

\begin{document}

\theoremstyle{plain}
\newtheorem{theorem}{Theorem}
\newtheorem{corollary}[theorem]{Corollary}
\newtheorem{lemma}[theorem]{Lemma}
\newtheorem{proposition}[theorem]{Proposition}

\theoremstyle{definition}
\newtheorem{definition}[theorem]{Definition}
\newtheorem{example}[theorem]{Example}
\newtheorem{conjecture}[theorem]{Conjecture}

\theoremstyle{remark}
\newtheorem{remark}[theorem]{Remark}

\begin{center}
\vskip 1cm{\LARGE\bf \LARGE\bf Generalized Compositions of Natural Numbers \\
\vskip .1in
}
\vskip 1cm
\large {
Milan Janji\'c \\
Department of Mathematics and Informatics\\
 University of Banja Luka \\
Republic of Srpska , BA}\\
\end{center}

\begin{abstract}
We consider compositions of natural numbers when  there are different types of each natural number.   Several recursions as well as some closed formulas for the number of compositions is derived.  We also find its  relationships with some known classes of integers such as Fibonacci, Catalan, Pell, Pell-Lucas, and Jacobsthal numbers.

\end{abstract}

\section{Introduction}

Let $\mathbf b=(b_1,b_2,\ldots)$ be a sequence of nonnegative integers.
      Compositions  of $n$ in which there are $b_1$ different type of $1$'s, $b_2$ different type of $2$'s, and so on, will be called
    generalized compositions of $n$. We let  $c(n,\mathbf b)$ denote its number.
We may considered these compositions as colored compositions in which each number $i$ may be colored by one of  $b_i$ colors.
If    all $b_i$ are equal $1$ then the standard compositions are obtained.

It is clear that the following recursion for $c(n,\mathbf b)$ holds
\[c(n,\mathbf b)=b_1c(n-1,\mathbf b)+b_2c(n-2,\mathbf b)+\cdots+b_{n-1}c(1,\mathbf b)+\cdots+b_n,\]
having $b_1$ generalized compositions ending by one of $1$'s, $b_2$ generalized compositions ending by one of $2$'s, and so on. At the end, there are $b_n$ generalized compositions consisting of one of $n$'s.

We define the sequence $(a_1,a_2,\ldots)$ such that $a_1=1,$ and
\begin{equation}\label{e1}a_{n+1}=\sum_{i=1}^{n}b_{n+1-i}a_{i}.\end{equation}

It is clear that \[a_{n+1}=c(n,\mathbf b),\;(n=1,2,\ldots).\]

Equation (\ref{e1})   connects two sequences of nonnegative integers
\[(b_1,b_2,\ldots),\mbox{ and }(c_1,c_2,\ldots),\] where $c_i=a_{i+1},\;(i=1,2,\ldots).$

Obviously, for each sequence $(b_1,b_2,\ldots)$ we may form the sequence $(c_1,c_2,\ldots).$ Conversely is not true.
 Namely, equation (\ref{e1}) may be regarded as a recurrence relation with respect to $b$'s, but it does not ultimately produce  nonnegative integers.

The paper is organized as follows. In this section we find a simple but interesting connection of generalized compositions with Catalan numbers.

  In Section 2 we consider the case when $b$'s make an arithmetical progression. We shall prove that then the numbers $c(n,\mathbf b)$ satisfy a three terms homogenous recursion with constant coefficients. This means that a close formula for generalized compositions  may be obtained.
In a particular case the number of generalized composition is a Pell-Lucas number.

In Section 3 we consider the case when $b_i$ is a square function of $i.$ Then the numbers $c(n,\mathbf b)$ satisfy a four terms homogenous recursion with constant coefficients. Thus, in this case also we may derive an explicit formula for generalized compositions. Special attention is put on triangular numbers. Several results will be obtained  in the case when  $b$'s are triangular numbers. Then  the $a$'s are sums of binomial coefficients. Some identities, concerning sums of binomial coefficients,  will be derived by the the use  of Zeilberger's algorithm, which is described by Petkovsek and all., in  \cite{pet}.

In Section 4 we investigate the case when $b_i$ is an exponential function of $i.$ In Section 5 two result concerning the floor and the ceil functions will be proved.

 We shall see, in Section 6,  that the generalized compositions are closely related with Fibonacci numbers, as is the case with the standard compositions. Several recurrence relations  as well as some closed formulas for generalized compositions will be proved. New relationships of Fibonacci numbers with Pell, Jacobsthal and other classes of numbers are derived.

Note that there is a significant number of sequences in Sloane's OEIS, \cite{slo}, which terms equal the number of generalized compositions. Comment of these sequences in OEIS offer other interpretations of compositions. Sequence {A145839} connects generalized compositions with so called matrix compositions. Also  {A020729}, {A008776}, {A020698}, {A007484} connect them with Pisot sequences.

\begin{proposition} If $\mathbf b=(p,p,\ldots),$ then \[c(n,\mathbf b)=p(1+p)^{n-1}.\]
Particulary, $2^{n-1}$ is the number of all compositions of $n.$
\end{proposition}
\begin{proof} In this case the recurrence (\ref{e1}) becomes.
\[a_{n+1}=p\sum_{i=1}^na_i.\]
Replacing $n$ by $n+1$ yields
\[a_{n+2}=p\sum_{i=1}^{n+1}a_i.\] By substraction we obtain
\[a_{n+2}=(1+p)a_{n+1}.\] From this equation the assertion follows easily.

\end{proof}
The next result shows that Catalan numbers give  an example when  $b$'s produce $b$'s again.

\begin{proposition}
If $\mathbf b=(C_0,C_1,\ldots),$ where $C_i,\;(i=0,1,\ldots)$ are Catalan numbers,  then
\[ c(n,\mathbf b)=C_n,\;(n\geq 1).\]
\end{proposition}
\begin{proof}
In this case equation (\ref{e1}) has the form:
\[a_{n+1}=\sum_{i=1}^{n}C_{n-i}a_{i}=\sum_{i=0}^{n-1}C_{n-i-1}a_{i+1}.\]

The equation  $a_{n+1}=C_n$ follows by induction, using the well-known Segner's recurrence formula for Catalan numbers.
\end{proof}

\section{Arithmetic Progressions}
In this section we consider the case when $b_i$ is a linear function of $i,$ that is, when  $b$'s make an arithmetic progression.

 We shall prove that then the numbers of generalized compositions satisfy a three terms recursion with constant coefficients. In this way the explicit formula for the number of compositions may be obtained.

\begin{proposition} Let $n$ be a positive integer, let $m,\;k$ be nonnegative integers, and let $b_i=m(i-1)+k,\;(i=1,2,\ldots).$
Then
\[c(1,\mathbf b)=k,\;c(2,\mathbf b)=m+k+k^2,\]
\[c(n+1,\mathbf b)=(k+2)c(n,\mathbf b)+(m-k-1)c(n-1,\mathbf b).\]

\end{proposition}
\begin{proof} Equation (\ref{e1}) takes the form: \[a_{n+1}=\sum_{i=1}^n[m(n-i)+k]a_i.\]
It is easy to see that  \[a_2=k,\;a_3=m+k+k^2.\]
Further, for $n>2$  we have \[a_{n+1}=ka_n+\sum_{i=1}^{n-1}[m(n-i)+k]a_i=ka_n+(m+k)a_{n-1}+\sum_{i=1}^{n-2}[m(n-1-i)+k]a_i+m\sum_{i=1}^{n-2}a_i.\]
We conclude that
\begin{equation}\label{a1}a_{n+1}=(k+1)a_n+ma_{n-1}+m\sum_{i=1}^{n-2}a_i.\end{equation}
Replacing $n$ by $n+1$ yields
\begin{equation}\label{a2}a_{n+2}=(k+1)a_{n+1}+ma_{n}+m\sum_{i=1}^{n-1}a_i.\end{equation}
Subtracting  equation (\ref{a1}) from (\ref{a2}) we obtain
\[a_{n+2}=(2+k)a_{n+1}+(m-k-1)a_n.\]
\end{proof}
In the next corollary we give some particular cases.

\begin{corollary}
\enumerate
\item[(i)]  If $m=1,\;k=0$ then   \[c(1,\mathbf b)=0,\;c(n,\mathbf b)=2^{n-2},\;(n>1).\]
\item[(ii)] If $m=2,\;k=0$ then   \[c(n,\mathbf b)=2\sum_{i=0}^{\lfloor\frac n2 \rfloor}{n\choose 2i}2^{i}.\]

\item[(iii)]  If $m=1,\;k=1$ then  \[c(n,\mathbf b)=F_{2n}.\]
\item[(iv)]
 If $k=m-1$ then
  \[c(1,\mathbf b)=m-1,\;c(n,\mathbf b)=m^2\cdot (m+1)^{n-2},\;(n>1).\]
\end{corollary}
\begin{proof}
$(i)$ is obvious.

In the case $(ii)$ the recurrence equation takes the form
\[c(1,\mathbf b)=0,\;c(2,\mathbf b)=2,\]
\[c(n+1,\mathbf b)=2c(n,\mathbf b)+c(n-1,\mathbf b),\]
which is the recurrence for Pell-Lucas numbers.

In the case $(iii)$ the recurrence becomes
\[c(1,\mathbf b)=1,\;c(2,\mathbf b)=3,\]
\[c(n+1,\mathbf b)=3c(n,\mathbf b)-c(n-1,\mathbf b).\]

This is the recurrence equation for Fibonacci numbers with even indices by Identity 7 from \cite{bk}.

Finally, for $k=m-1$ we have
\[c(1,\mathbf b)=m-1,\;c(2,\mathbf b)=m^2,\]
\[c(n+1,\mathbf b)=(m+1)c(n,\mathbf b),\;(n>2),\]
and $(iv)$ is true.
\end{proof}
\begin{remark} The formulas from the preceding corollary generate the following sequences from OEIS.

\begin{tabular}{ll} $m=1,\; k=0,$ {A000079};&$m=2,\; k=0,$ Pell-Lucas  numbers, {A052542};\\
 $m=1,\; k=1,$ {A001906}.
\end{tabular}

In the case $k=m-1$ we have

  \begin{tabular}{ccc}
  $m=2,$ {A003946};& $m=3,$ {A055841};& $m=4,$ {A055842};\\  $m=5,$ {A055846};& $m=6,$ {A055270};& $m=7,$ {A055847};\\ $m=8,$ {A055995};& $m=9,$ {A055996};& $m=10,$ {A056002};\\ $m=11,$ {A056116}.\end{tabular}
\end{remark}

\section{Square Functions}
In the case that $b_i$ is a square function od $i$ we obtain the four terms recurrence relation for $c(n,\mathbf b).$ This means that
we may obtained a closed formula for generalized compositions.

\begin{proposition}\label{tt1} Let $n$ be a positive integer, and let $k,m,p$ be arbitrary (rational) numbers such that  $b_i=ki^2+mi+p,\;(i=1,2,\ldots)$ are nonnegative integers. Then
\[c(1,\mathbf b)=k+m+p,\;c(2,\mathbf b)=k^2+m^2+p^2+2(km+mp+kp)+4k+2m+p,\]
\[\;c(3,\mathbf b)=8k+3m+p+2(4k^2+2m^2+p^2+6km+5kp+3mp)+\]
\[+k^3+m^3+p^3+3(km^2+kp^2+mp^2+2kmp).\]
\[c(n+1,\mathbf b)=(k+m+p+3)c(n,\mathbf b)+(k-m-2p-3)c(n-1,\mathbf b)+(p+1)c(n-2,\mathbf b),\;(n\geq 3).\]
\end{proposition}
\begin{proof} We have
\[a_{n+1}=\sum_{i=1}^n\big[k(n-i+1)^2+m(n-i+1)+p\big]a_i.\]
It is easy to obtain the values $a_2,a_3$ and $a_4.$
For $n>3$ we have
\[a_{n+1}=(k+p+m)a_n+\sum_{i=1}^{n-1}\big[k(n-i+1)^2+m(n-i+1)+p\big]a_i.\]
It follows that
\begin{equation}\label{q1}a_{n+1}=(k+m+p+1)a_n+(3k+m)a_{n-1}+
\sum_{i=1}^{n-2}\big[k(2n-2i+1)+m\big]a_i.\end{equation}
Replacing $n$ by $n+1$ we obtain
\begin{equation}\label{q2}a_{n+2}=(k+m+p+1)a_{n+1}+(3k+m)a_{n}+
\sum_{i=1}^{n-1}\big[k(2n-2i+1)+m\big]a_i+2k\sum_{i=1}^{n-1}a_i.\end{equation}
Subtracting (\ref{q1}) from (\ref{q2}) yields
\begin{equation}\label{q3}a_{n+2}=(k+m+p+2)a_{n+1}+(2k-p-1)a_{n}+2k\sum_{i=1}^{n-1}a_i.\end{equation}
 Replacing $n$ by $n+1$ we obtain
\begin{equation}\label{q4}a_{n+3}=(k+m+p+2)a_{n+2}+(2k-p-1)a_{n+1}+2k\sum_{i=1}^{n}a_i.\end{equation}
Finally, subtracting (\ref{q3}) from (\ref{q4}) yields
\[a_{n+3}=(k+m+p+3)a_{n+2}+(k-m-2p-3)a_{n+1}+(p+1)a_n.\]
\end{proof}
In the next corollary we give two particular cases.
\begin{corollary}
\enumerate
\item[(i)]  If $k=1,m=0,p=-1$ then   \[c(1,\mathbf b)=0,c(2,\mathbf b)=3,\]
    \[c(n,\mathbf b)=8\cdot 3^{n-3},\;(n\geq 3).\]

\item[(ii)] If $k=1,m=1,p=-1$ then   \[c(1,\mathbf b)=1,c(2,\mathbf b)=6,c(3,\mathbf b)=22,\]
    \[c(n,\mathbf b)=\frac{9-5\sqrt 3}{6}(2+\sqrt 3)^n+\frac{9+5\sqrt 3}{6}(2-\sqrt 3)^n.\]
\end{corollary}\label{cor1}
\begin{proof}
The assertion $(i)$ is true since, in this case, the recurrence equation becomes \[c(n+1,\mathbf b)=3c(n,\mathbf b),\;(n\geq 2).\]
In the case $(ii)$ the recurrence takes the form:
\[c(n+1,\mathbf b)=4c(n,\mathbf b)-c(n-1,\mathbf b),\;(n\geq 3).\]
Solving the characteristic equation of this three terms recurrence equation we conclude that the assertion is true.
\end{proof}
\begin{remark} We state two sequences from OEIS generated by the preceding formulas.
$k=1,\;m=0,\;p=-1$, {A118264},  $k=1,\;m=1,\;p=-1$,  {A003699}.
\end{remark}

Since $n\choose 2$ is a square function of $n$ we may derive from the preceding proposition some formulas which connect triangular numbers with generalized compositions.
\begin{corollary}
\enumerate
\item[(i)]  If $b_i={i-2\choose 2},\;(i=1,2,\ldots)$ then
\[c(1,\mathbf b)=1,c(2,\mathbf b)=1,c(3,\mathbf b)=1,\]
\[c(n+1,\mathbf b)=4c(n,\mathbf b)-6c(n-1,\mathbf b)+4c(n-2,\mathbf b),\;(n\geq 3).\]
Explicitly,
\[c(n,\mathbf b)=\sum_{i=0}^{n}{n\choose 4n-4i}.\]

\item[(ii)] If $b_i={i-1\choose 2},\;(i=1,2,\ldots)$ then
\[c(1,\mathbf b)=0,c(2,\mathbf b)=0,c(3,\mathbf b)=1,\]
\[c(n+1,\mathbf b)=3c(n,\mathbf b)-3c(n-1,\mathbf b)+2c(n-2,\mathbf b),\;(n\geq 3).\]
Explicitly,
\[c(n,\mathbf b)=\sum_{i=0}^{\lfloor\frac{n-3}{3}\rfloor}{n-1\choose 3i+2}.\]

\item[(iii)]  If $b_i={i\choose 2},\;(i=1,2,\ldots)$ then
\[c(1,\mathbf b)=0,c(2,\mathbf b)=1,c(3,\mathbf b)=3,\]
\[c(n+1,\mathbf b)=3c(n,\mathbf b)-2c(n-1,\mathbf b)+c(n-2,\mathbf b),\;(n\geq 3).\]
Explicitly,
\[c(n,\mathbf b)=\sum_{i=0}^{n}{n+i\choose 3i+2}.\]

\item[(iv)]  If $b_i={i+1\choose 2},\;(i=1,2,\ldots)$ then
\[c(1,\mathbf b)=1,c(2,\mathbf b)=4,c(3,\mathbf b)=13,\]
\[c(n+1,\mathbf b)=4c(n,\mathbf b)-3c(n-1,\mathbf b)+c(n-2,\mathbf b),\;(n\geq 3).\]
Explicitly,
\[c(n,\mathbf b)=\sum_{i=0}^{n}{n+2i-1\choose n-i}.\]
\end{corollary}

\begin{proof}
  The assertion $(i)$ is obtained for  $k=\frac 12,\;m=-\frac 52,\;p=3.$

  The assertion $(ii)$ is obtained for  $k=\frac 12,\;m=-\frac 32,\;p=1.$

The assertion $(iii)$ is obtained for  $k=\frac 12,\;m=-\frac 12,\;p=0.$

The assertion $(iv)$ is obtained for  $k=\frac 12,\;m=\frac 12,\;p=0.$

The explicit formulas are obtained by the use of Zeilberger's algorithm, \cite{pet}.
\end{proof}

\begin{remark} In the case $b_i={i+2\choose 2},\;(i=1,2,\ldots)$
we obtain {A145839} which counts the number of $3$-compositions of $n.$ This connects our compositions with the so called matrix compositions.

The following sequences in OEIS are generated by the preceding formulas.

\begin{tabular}{cccc} (i), {A038503};&(ii), {A024495};& (iii), {A095263};&
(iv), {A095263}.\end{tabular}
\end{remark}
\section{Exponential Functions}
The following result concerns the case when $b_i$ is an exponential functions of $i.$ Then, again, the numbers $c(n,\mathbf b)$ satisfy a three terms homogeneous recurrence relation with constant coefficients.
\begin{proposition} Let $n$ be a positive integer. If $b_i=k+pm^{i-1},\;(i=1,2,\ldots),$ then
\[c(1,\mathbf b)=k+p,\;c(2,\mathbf b)=k+pm+(k+p)^2,\]
\[c(n,\mathbf b)=(k+m+p+1)c(n-1,\mathbf b)-(km+m+p)c(n-1,\mathbf b),\;(n>2).\]
\end{proposition}
\begin{proof} Equation (\ref{e1}) has the form:
\[a_{n+1}=\sum_{i=1}^n(k+pm^{n-i})a_i.\]
It follows that
\[a_2=k+p,\;a_3=k+pm+(k+p)^2.\]
Further we have
\[a_{n+1}=(k+p)a_n+\sum_{i=1}^{n-1}(k+pm^{n-i})a_i=(k+p)a_n+k\sum_{i=1}^{n-1}a_i+p\sum_{i=1}^{n-1}m^{n-i}a_i.\]
Hence,
\[a_{n+1}=
(k+p)a_n+(k+p)a_{n-1}+k\sum_{i=1}^{n-2}a_i+pm\sum_{i=1}^{n-2}m^{n-1-i}a_i+p(m-1)a_{n-1},\] that is,
\begin{equation}\label{ppp1}a_{n+1}=(k+p+1)a_n+p(m-1)a_{n-1}+p(m-1)\sum_{i=1}^{n-2}m^{n-1-i}a_i.\end{equation}
Replacing $n$ by $n+1$ we obtain
\begin{equation}\label{ppp2}a_{n+2}=(k+p+1)a_{n+1}+p(m-1)a_{n}+pm(m-1)\sum_{i=1}^{n-1}m^{n-1-i}a_i.\end{equation}
Subtracting (\ref{ppp1}) multiplied by $m$ from (\ref{ppp2}) we obtain
\[a_{n+2}=(k+p+m+1)a_{n+1}-(p+m+mk)a_n.\]
\end{proof}
Some particular cases of the preceding proposition follow.
\begin{corollary}
\enumerate
\item[(i)] If $k=0$ then \[c(n,\mathbf b)=p(m+p)^{n-1}.\]
\item[(ii)] If $k=1,m=2,p=1$ then
\[c(1,\mathbf b),\;c(2,\mathbf b)=7,\]
\[c(n,\mathbf b)=5c(n-1,\mathbf b)-5c(n-2,\mathbf b),\;(n>2).\]
\item[(iii)]If $k=-1,m=2,p=1$ then
\[c(n,\mathbf b)=F_{2n-2}.\]

\begin{proof} In the case $k=0$ we have
\[c(1,\mathbf b)=p,\;c(2,\mathbf b)=p(m+p),\]
\[c(n,\mathbf b)=(m+p+1)c(n-1,\mathbf b)-(m+p)c(n-1,\mathbf b),\;(n>2).\]
The roots of the characteristic equation are $\alpha=1,\;\beta=m+p.$
Solving the system
\[c_1\alpha+c_2\beta=p,\;c_1\alpha^2+c_2\beta^2=p(m+p),\] yields
$c_1=0,\;c_2=\frac{p}{m+p},$ and the assertion $(i)$ is true.

In the case $(ii)$ we have
\[c(1,\mathbf b)=2,\;c(2,\mathbf b)=7,\]
\[c(n,\mathbf b)=5c(n-1,\mathbf b)-5c(n-2,\mathbf b),\;(n>2).\]

Finally, in the case $(iii)$ we have
\[c(1,\mathbf b)=0,\;c(2,\mathbf b)=1,\]
\[c(n,\mathbf b)=3c(n-1,\mathbf b)-c(n-2,\mathbf b),\;(n>2).\]
The assertion follows from Identity 7 in \cite{bk}.
\end{proof}
\end{corollary}

\begin{remark}

We state sequences in OEIS defined with $k=0,\;p=1,$ and $m$ ranges  from $2$ to $39.$

\noindent
{A000244},
{A000302},
{A000351},
{A000400},
{A000420},
{A001018},
{A001019},
{A011557},
{A001020},
{A001021},
{A001022}.
{A001023},
{A001024},
{A001025},
{A001026},
{A001027},
{A001029},
{A009964},
{A009965},
{A009966},
{A009967},
{A009968},
{A009969},
{A009970},
{A009971},
{A009972},
{A009973},
{A009974},
{A009975},
{A009976},
{A009977},
{A009978},
{A009979},
{A009980},
{A009981},
{A009982},
{A009983},
{A009984}.

More sequences follow

\begin{tabular}{ccc}
k=0, m=2, p=3,  {A005053};&
k=0, m=2, p=2, {A081294};&
k=1, m=2, p=1, {A052936};\\
k=1, m=3, p=1, {A034999};&
k=0, m=2, p=4, { A067411};&
k=0, m=3, p=2, {A020729};\\
k=0, m=4, p=2, {A167747};&
k=0, m=4, p=3, {A169634};&
k=0, m=4, p=5, {A067403};\\
k=0, m=4, p=6, {A090019};&
k=0, m=5, p=2, {A109808};&
k=0, m=5, p=3, {A103333};\\
k=0, m=6, p=2, {A013730};&
k=0, m=6, p=3, {A013708};&
k=0, m=6, p=4, {A093141};\\
k=2, m=2, p=1, {A163606};&
k=-1, m=2, p=1, {A001906};&
k=-2, m=2, p=1, {A001333};\\
k=-1, m=3, p=1, {A052530}.
\end{tabular}
\end{remark}

\section{Floor and Ceil Functions } In this section we derive two results when $b_i$ is a floor, and a ceil function of $i.$
\begin{proposition} Let $n$ be a positive integer, and $b_i=\lfloor\frac i2\rfloor,\;(i=1,2,\ldots).$ Then
  \[c(1,\mathbf b)=0,\;c(2,\mathbf b)=1,\;c(3,\mathbf b)=1,\]
  \[c(n,\mathbf b)=c(n-1,\mathbf b)+2c(n-2,\mathbf b)-c(n-3,\mathbf b),\;(n>3).\] \end{proposition}
 \begin{proof}
 It is easy to see that
\[a_2=0,\;a_3=a_4=1.\]

 For $n>3$  we have
 \[a_{n+1}=\sum_{i=1}^{n-1}\left\lfloor\frac{ n-i+1}{2}\right\rfloor a_i=a_{n-1}+\sum_{i=1}^{n-2}\left\lfloor\frac{ n-i+1}{2}\right\rfloor a_i=\]\[=a_n+a_{n-1}+\sum_{i=1}^{n-2}\left\{\left\lfloor\frac{ n-i+1}{2}\right\rfloor -\left\lfloor\frac{ n-1-i+1}{2}\right\rfloor\right\}a_i.\]
 It follows that
 \begin{equation}\label{f1}a_{n+1}=a_n+a_{n-1}+\sum_{i=1}^{n-2}\frac{(-1)^{n-i+1}+1}{2}a_i.\end{equation}
Replacing $n$ by $n+1$ we obtain
 \begin{equation}\label{f2}a_{n+2}=a_{n+1}+a_{n}-\sum_{i=1}^{n-1}\frac{(-1)^{n-i+1}+1}{2}a_i+\sum_{i=1}^{n-1}a_i.\end{equation}
Substituting (\ref{f1}) into (\ref{f2}) yields
 \begin{equation}\label{f3}a_{n+2}=2a_{n}+\sum_{i=1}^{n-1}a_i.\end{equation}
Replacing $n$ by $n+1$ we have
\begin{equation}\label{f4}a_{n+3}=2a_{n+1}+\sum_{i=1}^{n}a_i.\end{equation}
Subtracting (\ref{f3}) from (\ref{f4}) yields
\[a_{n+3}=a_{n+2}+2a_{n+1}-a_n.\]
 \end{proof}
 \begin{remark}
  Sequence {A006053} is generated by this function.
\end{remark}
In a similar way the following proposition may be proved:
\begin{proposition} Let $n$ be a positive integer, and $b_i=\lceil\frac i2\rceil,\;(i=1,2,\ldots).$ Then
  \[c(1,\mathbf b)=1,\;c(2,\mathbf b)=2,\;c(3,\mathbf b)=5,\]
  \[c(n,\mathbf b)=2c(n-1,\mathbf b)+c(n-2,\mathbf b)-c(n-3,\mathbf b),\;(n>3).\] \end{proposition}

\begin{remark}
  Sequence {A006054} is generated by this function.
\end{remark}

\section{Fibonacci numbers }
 In this section  we prove several formulas in which the number of compositions is related with Fibonacci numbers.

Our first result extends the result from  \cite{eme}, where
compositions with two different types of $1$ are considered, as
well as some other known results about standard compositions.
\begin{proposition} Let $n$  be a positive integer, let $p,\;q$ be a nonnegative integers, and let $\mathbf b=(p,q,q,\ldots).$
Then \[c(1,\mathbf b)=p,\;c(2,\mathbf b)=p^2+q,\]
\begin{equation}\label{e2}c(n,\mathbf b)=(1+p)c(n-1,\mathbf b)+(q-p)c(n-2,\mathbf b),\;(n>2).\end{equation}
Explicitly,
\[c(n,\mathbf b)=u\alpha^n+v\beta^n,\]
where
\[\alpha=\frac{1+p+\sqrt{(p-1)^2+4q}}{2},\;\beta=\frac{1+p-\sqrt{(p-1)^2+4q}}{2},\] \[u= \frac{4q-(p-1)^2+(p-1)\sqrt{(p-1)^2+4q}}{2\sqrt{(p-1)^2+4q}},\]
\[v= \frac{4q+(p-1)^2-(p-1)\sqrt{(p-1)^2+4q}}{2\sqrt{(p-1)^2+4q}}.\]
\end{proposition}

\begin{proof}
In this case equation (\ref{e1}) has the form:
\[a_{n+1}=pa_n+q\sum_{i=1}^{n-1}a_{i}.\]
 We easily obtain that
 \[a_2=p,\;a_3=p^2+q.\] Next we have
    \[a_{n+1}=pa_n+qa_{n-1}-pa_{n-1}+pa_{n-1}+q\sum_{i=1}^{n-2}a_{i}=\]\[
=(1+p)a_n+(q-p)a_{n-1}.\]

The characteristic equation for this recurrence  is $x^2-(p+1)x-q+p=0.$ Solving this equation we obtain the explicit formula.
\end{proof}

In the following corollary we shall state some particular cases of this proposition.
 The first is the well-known formula for the number of all standard compositions. In the rest Fibonacci numbers are produced.
\begin{corollary}
\begin{itemize}
\item[$(i)$] If $\mathbf b=(0,1,1,\ldots)$
then $c(n,\mathbf b)=F_{n-1}.$ This is the well-known result that
says that there are $F_ {n-1}$ compositions of $n$ in
which each part is $\geq 2.$
\item[$(ii)$] If $\mathbf
b=(2,1,1,\ldots)$ then $c(n,\mathbf b)=F_{2n+1}.$ This is the
result from \cite{eme}.

\item[$(iii)$]  If $\mathbf b=(3,4,4,\ldots)$ then
\[c(n,\mathbf b)=F_{3n+1}.\]
\end{itemize}
\end{corollary}
\begin{proof}

(i). In this case we have  $p=0,\;q=1.$ It follows that
\[\alpha=\frac{1+\sqrt 5}{2},\;\beta=\frac{1-\sqrt 5}{2},
u=\frac{5-\sqrt 5}{10}, \;v=\frac{5+\sqrt 5}{10},\] and the assertion follows from Binet formula.

 (ii). In this case we have $p=2,\;q=1$ and  the recurrence relation has the form
 \[a_2=2,\;a_3=5,\;a_{n+1}=3a_n-a_{n-1}.\]
 The assertion follows by induction using Identity 17 from \cite{bk}.

(iii). The recurrence equation in this case has the form
\[a_{n+1}=4a_n+a_{n-1}.\] It is easy to prove that for Fibonacci numbers the following identity holds
 \[F_{k+2}=4F_{k-1}+F_{k-4}.\]

Using induction and this identity we conclude that the assertion holds.
\end{proof}

\begin{remark} We state several sequences from OEIS which are generated by (\ref{e2}).

\[\begin{tabular}{lll}
p=2, q=1,{ A001519};&p=3, q=1 {A007052};&p=4, q=1, {A018902};\\
p=5, q=1, {A018903};&p=6, q=1, {A018904};&p=1, q=2, {A001333};\\
p=1, q=3, {A026150};&p=1, q=4, {A046717};&p=1, q=5, {A084057};\\
p=1, q=6, {A002533};&p=1, q=7, {A083098};&p=1, q=8, {A083100};\\
p=1 ,q=9, {A003665};&p=1, q=10, {A002535};&p=1, q=11, {A083101};\\
p=1, q=12, {A090042};&p=1, q=13, {A125816};&p=1, q=14, {A133343};\\
p=1, q=15, {A133345};&p=1, q=16, {A120612};&p=1, q=17, {A133356};\\
p=1, q=18, {A125818};&p=2, q=3, {A052924};&p=2, q=4, {A104934};\\
p=2, q=6, {A122117};&p=3, q=2, {A001835};&p=3, q=4, {A033887};\\
p=3, q=6, {A122558};&p=3, q=8, {A083217};&p=3, q=9, {A147518};\\
p=4, q=2, {A052913};&p=4, q=3, {A004253};&p=4, q=5, {A100237};\\
p=5, q=2, {A158869};&p=5, q=4, {A001653}.&
\end{tabular}\]
\end{remark}

The next result also generalizes a classical result for standard
compositions.
\begin{proposition} If $\mathbf b=(p,1,0,0,\ldots)$ where $m$ is a positive integer then
\[c(n,\mathbf b)=\sum_{i=0}^{\lfloor\frac{n}{2}\rfloor}{n-i\choose i}p^{n-2i}.\]

\end{proposition}
\begin{proof}
The recurrence relation for Fibonacci polynomials is
\[F_{n+1}(x)=xF_n(x)+F_{n-1}(x).\] It follows that for a positive
integer $p$ we have $F_{n+1}(p)=c(n,\mathbf b).$ The required
equation follows from the well-known formula for Fibonacci
polynomials.

\end{proof}
As an immediate consequence we obtain the following well-known
result.
\begin{corollary} The number of compositions of $n$ in which each part is either $1$ or $2$ is $F_{n+1}.$
\end{corollary}
\begin{proof} Take $p=1$ in the preceding proposition and  apply Identity 4 in \cite{bk}.
\end{proof}

\begin{remark}
As before, we state a few sequences from OEIS obtained for different values
of $p.$

\begin{tabular}{lll} p=3,\; {A000129}. Pell numbers,& p=4,\;{A006190},& p=5,\; {A001076},\\
p=6,\; {A052918},&p=7,\; {A054413}.&
\end{tabular}
\end{remark}

The following result also extends a well-known result for standard compositions.

\begin{proposition}
Let $n$ be a positive integer, let $p,\;q$ be nonnegative integers, and let  $\mathbf b=(p,q,p,q,\ldots).$
Then \[ c(1,\mathbf b)=p,\;c(2,\mathbf b)=p^2+q,\]\[c(n,\mathbf b)=pc(n-1,\mathbf b)+
(1+q)c(n-2,\mathbf b),\;(n>2).\]
\end{proposition}
\begin{proof}
In this case we first have \[a_{2n+1}=\sum_{i=1}^{2n}b_{2n-i+1}a_i.\]
Hence,
\[a_{2m+1}=q(a_1+a_3+\cdots+a_{2m-1})+p(a_2+a_4+\cdots+a_{2m})=\]
\[=qa_{2m-1}+pa_{2m}+a_{2m-1}=pa_{2m}+(1+q)a_{2m-1},\] and the assertion is true for even $n.$
Also,
\[a_{2m}=\sum_{i=1}^{2m-1}b_{2m-i}a_i,\] that is
\[a_{2m}=p(a_1+a_3+\cdots+a_{2m-1})+q(a_2+a_4+\cdots+a_{2m-2})=\]
\[=pa_{2m-1}+qa_{2m-2}+a_{2m-2}=pa_{2m-1}+(1+q)a_{2m-2}.\] Hence, the assertion is also true for odd $n.$
\end{proof}
\begin{corollary} Let $n$ be a positive integer, and let $\mathbf b=(1,0,1,0,\ldots).$ Then
\[c(n,\mathbf b)=F_n.\] In other word, $F_n$ is the number of
compositions of $n$ in which all parts are odd.
\end{corollary}
\begin{proof}
Since $p=1,\;q=0$ the recurrence from the preceding propositions becomes recurrence relation for Fibonacci numbers.
\end{proof}

\begin{remark}
The following sequences from OEIS are generated by the formula from this proposition.

p=1, q=2, {A105476}, p=2,q=1, {A052945}, p=2,q=3, {A162770}.

\end{remark}

In the rest of this section Fibonacci numbers play the role of the $b$'s.

We shall prove that there are a closed formula for $c(n,\mathbf b)$ in the case when \[b_{i}=F_{m+k(i-1)},\;(i=1,2,\ldots,n)\] where $m\geq -1$ and $k\geq 0$ are  arbitrary integers.  For this we need the following identities for Fibonacci numbers.
\begin{lemma} Let $m\geq -1,\;k\geq 0$ be integers. Then
\begin{equation}\label{pp1}F_{m+2k}+F_{m-2k}=F_{m}(F_{2k-1}+F_{2k+1}).\end{equation}
Also,
\begin{equation}\label{pp2}F_{m+2k-1}-F_{m-2k+1}=F_{m}(F_{2k-2}+F_{2k}).\end{equation}
\end{lemma}
\begin{proof}
The assertion (\ref{pp1}) is obviously true for $m=0.$  Since  $F_{-(2k-1)}=F_{2k-1}$ it is also  true for $m=1$ and $m=-1.$
Assume that it is true for $m_1$ such that $0\leq m_1<m.$
Then, for $m\geq 2$ we have
\[F_{m+2k}+F_{m-2k}=F_{m-1+2k}+F_{m-1-2k}+F_{m-2+2k}+F_{m-2-2k}.\]
Using the induction hypothesis yields
\[F_{m+2k}+F_{m-2k}=(F_{m-1}+F_{m-2})(F_{2k-1}+F_{2k+1})=F_{m}(F_{2k-1}+F_{2k+1}).\]
The assertion (\ref{pp2}) may be proved in a similar way.
\end{proof}
\begin{proposition} Let $m\geq -1$ be an integer, let $k$ be a nonnegative integer, and let  $b_i=F_{m+k(i-1)},\;(i=1,2,\ldots,n ).$
Then,
\[c(1,\mathbf b)=F_m,\;c(2,\mathbf b)=F_{m+k}+F_m^2,\]
\[c(n+1,\mathbf b)=(F_{m}+F_{k-1}+F_{k+1})c(n,\mathbf b)+(-1)^{k-1}(F_{m-k}+1)c(n-1,\mathbf b),\;(n>1).\]

\end{proposition}
\begin{proof}
We have
\[a_1=1,\;a_{n+1}=\sum_{i=1}^nF_{m+k(n-i)}a_i.\]
For $n=1,2$ we easily  obtain
\[a_2=F_m,\;a_3=F_{m+k}+F_m^2.\]
Assume that $k$ is even and denote $k=2p.$  Then for $n>2$ we obtain
\[a_{n+1}=\sum_{i=1}^nF_{m+2p(n-i)}a_i.\]

Using (\ref{pp1}) yields
\[(F_{2p-1}+F_{2p+1})a_{n+1}=\sum_{i=1}^nF_{m+2p(n+1-i)}a_i+\sum_{i=1}^nF_{m+2p(n-1-i)}a_i=\]
\[=a_{n+2}-F_ma_{n+1}+a_n+F_{m-2k}a_n,\] and the assertion holds.

If $k$ is odd and $k=2p-1,$ then for $n>2$ we have
\[a_{n+1}=\sum_{i=1}^nF_{m+(2p-1)(n-i)}a_i.\]

Using (\ref{pp2}) yields
\[(F_{2p-2}+F_{2p})a_{n+1}=\sum_{i=1}^nF_{m+(2p-1)(n+1-i)}a_i-\sum_{i=1}^nF_{m+(2p-1)(n-1-i)}a_i=\]
\[=a_{n+2}-F_ma_{n+1}-a_n-F_{m-2p+1}a_n,\] and the assertion also holds in this case.
\end{proof}
Particulary, for $k=0$ we have
\begin{corollary} Let $m\geq -1$ be an integer, and let $\mathbf b=(F_m,F_m,\ldots).$ Then
\[c(n,\mathbf b)=F_m(1+F_m)^{n-1},\;(n=1,2,\ldots).\]
\end{corollary}
 The preceding equation generalizes the formula for the number of all standard composition of $n$ which is obtained for $m=1.$
\begin{remark}
The preceding formula generates the following sequences in OEIS

\begin{tabular}{rcc}m=1 or m=2, {A000079},&m=3, {A008776},&m=4, {A002001}\\
m=5, {A052934},& m=6, {A055275}.\end{tabular}
\end{remark}

\begin{remark}
We again state some  sequences from OEIS generated with the formula from the preceding proposition.

\begin{tabular}{ll}m=0,\;k=1 {A001045} (Jacobsthal numbers)&m=1,\;k=1 { A000129}, ( Pell numbers),\end{tabular}

\begin{tabular}{lll}
m=2, k=1, {A028859};&m=3, k=1, {A007484};&m=-1, k=2, {A007051};\\ m=0, k=2, {A000244};&
m=1, k=2, {A007052};&m=2, k=2, {A001353};\\m=3, k=2, { A020698};&m=-1, k=3, {A147722};&
m=1, k=3, { A005054};\\m=3, k=3, { A078469}.&&\end{tabular}

\end{remark}
For our last result we need the following forth terms
recursion  for squares of Fibonacci numbers:
\begin{lemma} The following equation holds
\[F_{n+3}^2=2F_{n+2}^2+2F_{n+1}^2-F_n^2.\]
\end{lemma}
\begin{proof}

The formula is easy to prove by squaring the expressions $F_{n+3}=2F_{n+1}+F_n$ and $F_{n+2}=F_{n+1}+F_n.$

\end{proof}
\begin{proposition} Let $n$ be a positive integer, and let
$b_i=F_{k+i-1}^2,\;(i=1,2,\ldots).$ Then
\[c(1,\mathbf b)=F_k^2,\;c(2,\mathbf b)=F_{k+1}^2+F_k^4,\;c(3,\mathbf b)=F_{k+2}^2+2F_{k}^2F_{k+1}^2+F_k^6,\]
and, for $n>3,$
\[c(n,\mathbf b)=(F_k^2+2)c(n-1,\mathbf b)+(2F_{k-1}^2-F_{k-2}^2+2)c(n-2,\mathbf b)-(F_{k-1}^2+1)c(n-3,\mathbf b).\]
\end{proposition}
\begin{proof}
In this case we have
\[a_{n+1}=\sum_{i=1}^nF_{k+n-i}^2a_i.\]
Using the preceding lemma  yields
\[a_{n+1}=2\sum_{i=1}^nF_{k+n-1-i}^2a_i+2\sum_{i=1}^nF_{k+n-2-i}^2a_i-\sum_{i=1}^nF_{k+n-3-i}^2a_i=\]
\[=2F_{k-1}^2a_n+2a_n+2F_{k-2}^2a_n+2F_{k-1}^2a_{n-1}+2a_{n-1}-F_{k-3}^2a_n-\]\[-F_{k-2}^2a_{n-1}-F_{k-1}^2a_{n-2}-a_{n-2}.\]
Using the preceding lemma once more we obtain
\[a_{n+1}=(F_k^2+2)a_n+(2F_{k-1}^2-F_{k-2}^2+2)a_{n-1}-(F_{k-1}^2+1)a_{n-2},\]
and the assertion is proved.
\end{proof}
\begin{remark} The following two sequences in OEIS are generated by the preceding formula:\\
m=0, {A054854,} m=1, { A030186}.
\end{remark}

\bigskip
\hrule
\bigskip

\noindent 2000 {\it Mathematics Subject Classification}: Primary
11P99; Secondary 11B39.

\noindent \emph{Keywords: } composition of a natural number, Fibonacci number, Catalan number.

\end{document}